\documentclass[12pt]{amsart}

\usepackage{amssymb,amsmath,amsthm,epsfig,amscd,amssymb,latexsym,graphicx, epic, eepic}
\usepackage{hyperref}

\UseRawInputEncoding
\newtheorem{theorem}{Theorem}[section]
\newtheorem{lemma}[theorem]{Lemma}
\newtheorem{prop}[theorem]{Proposition}

\newtheorem{defn}{Definition}[section]

\def \mR{\mathbb{R}}

\def\R3{\mathbb{R}^3}

\def\Div{{\rm Div}\,}

\def \Ss{\mathbb {S}}

\def \E{\mathcal{E}}

\def \O{\mathcal{O}}

\def \ba {\begin {eqnarray*} }
\def \ea {\end {eqnarray*} }
\def \beq {\begin {eqnarray}}
\def \eeq {\end {eqnarray}}
\newcommand{\supp}{\mathrm{supp}}

\renewcommand{\phi}{\varphi}

\theoremstyle{definition}
\newtheorem{thm}{Theorem}[section]

\numberwithin{equation}{section}

\newcounter{sidenote}

\begin{document}
\centerline{\Large Unique Determination of a Planar Screen in Electromagentic}
\smallskip
\centerline{\Large Inverse Scattering}
\bigskip
\bigskip

\date
\maketitle

\centerline{\large Petri Ola\footnote{Department of Mathematics and Statistics, 00014 University of Helsinki}, Lassi P\"aiv\"arinta\footnote{Department of Mathematics, Tallinn University of Technology, Estonia} and Sadia Sadique\footnote{Department of Mathematics, Tallinn University of Technology, Estonia}\footnote{Reserach of P.O was supporeted by the Academy of Finland under CoE grant ..., Reasearch of L.P and S.S was supported by ERC grant ...}}
\bigskip

{\bf Abstract.} We show that the far--field pattern of a scattered electromagnetic field corresponding to a single incoming plane--wave uniquely determines a bounded supraconductive planar screen. This generalises a previous scalar result of Bl\aa sten, P\"aiv\"arinta and Sadique.
\medskip


%{\bf Keywords:} Inverse conductivity problem, electrical impedance
%tomography, unknown boundary, Teichm\"uller mapping.
\bigskip

\section*{Introduction}

The study of wave scattering from thin and large objects is rooted in the field of antenna theory. This interest began with a competition initiated by the Prussian Academy in 1879 to demonstrate the existence or non-existence of electromagnetic waves. James Clerk Maxwell's \cite{MJC} predicted the existence of such  waves by his mathematical theory 15 years earlier. In 1882, the competition was won by Heinrich Hertz, who used a dipole antenna to radiate and measure EM waves, thereby confirming Maxwell's theory. This experiment, along with Maxwell's theory, has had a profound impact on modern society. 

%Hertz's antenna was made up of two identical planar bodies, in his case squares, which produced radiating EM waves. As radiating antennas are identical to receiving antennas. The theory of antennas is closely linked to EM scattering and inverse scattering theory.

\medskip

Recently, there has been a growing interest in solving the inverse scattering problem with reduced measurements. Inverse scattering problems involves determining the properties of an object, such as its shape, size, and composition, from the scattered field produced when it is illuminated with electromagnetic waves. 
One way to determine the shape of an object is by analysing the pattern of waves that are scattered by the object in the far field. This is a problem that involves both mathematics and numerical techniques, and it has many practical applications. An overview of this topic can be found in the book \cite{CR} by Colton and Kress. In \cite{BPS} we studied the problem of fixed frequency acoustic scattering from a sound soft flat screen. In the above paper, our main result was that, given the plane where the screen is located, the far field produced by single incident waves determines the exact shape of the screen, given that it is not anti-symmetric with respect to the
plane. Our current work is the generalization of the result of 
\cite{BPS} to Maxwell's equations. 

%Sylvester and Uhlmann in \cite{SJG} introduced the method of using complex geometrical optics solutions and infinitely many far-field measurements in the fixed frequency setting. This was the first technique that allowed for the unique determination of an arbitrary smooth-enough scattering potential through far-field measurements. This method has since become a popular area of research, and the field has expanded rapidly.

\medskip

The shape determination problem in literature \cite{CR} is known as Schiffer'ss problems. The first uniqueness result for the case of the Dirichlet problem was presented by Schiffer in 1967 \cite{LDS}. The uniqueness theorem for the inverse Dirichlet problem, given by Schiffer \cite{LDS} requires information about the far-field of an infinity of incident frequencies. Due to Schiffer's uniqueness result for sound soft general obstacles by countably many incident plane waves(\cite{CR, LDS}) there has been an extensive study in this direction, see e.g.,
the result of uniqueness in the general domain in \cite{KR, DG,HJ, RA1,RA2,RL,SB}, for polyhedral scatterers \cite{JM, EM1, EM2, AGR}, for ball or disc \cite{LC,LZ} and for smooth planer curves \cite{KR1,KR2,ML,RL2}.  
In \cite{HLZ} the work has done on inverse electromagnetic scattering problem in the TE polarization case. They show that the knowledge of the electric far-field pattern for a single incoming wave is sufficient to uniquely determine the shape of rectangular penetrable scatterer.
See also Liu and Zou \cite{HJ}, who review some of the uniqueness results in inverse acoustic and electromagnetic obstacle scattering problems. There emphasis is given to the recent progress on the unique determination of general polyhedral scatterers by the far field data corresponding to one or several incident fields. For recent results in the time--harmonic inverse EM-scattering see the short review by Rainer Kress \cite{KR}. 
   
\bigskip
A  particular case in \cite{AH} gives the unique determination of a flat screen by a
single incident plane-wave measurement under some additional assumptions on the wave. To our knowledge
there is no proof for the unique determination of a planar screen
shape by one far-field pattern without restrictive a priori
assumptions. The results in \cite{NGM} comes very close: the obstacle
can be any Lipschitz domain as long as its boundary is not an analytic
manifold. But they does not considered screens , which is our priority.

\medskip
The goal of this work is to find unique determination of supporting hyperplane corresponding to the single measurement having non-vanishing far-fields. The proof follows from standard representation formula for exterior solution of Maxwell's equationss. The method we are using are based on the ideas of certain integral operator \cite{LS,PR}.  Similar to \cite{BPS} the crucial part of our argument is that the shape of the screen is exactly the support of the jump of the tangential component of the scattered magnetic field. Our main result is the following:

\begin{thm}\label{mainthm}
Let \(S\) be a \(C^2\)--screen contained in a supporting hyperplane \(L\) and let 
\[
E(\theta; p,q) = \mu^{1/2} (p\times \theta)e^{ik\langle \theta, x\rangle}, \, H(\theta; p,q) = \varepsilon^{1/2} (q\times \theta) e^{ik\langle \theta, x\rangle}
\]
the EM-plane wave with wavenumber \(k\), propagation direction \(\theta\) and polarizations \(p\) and \(q\). Let \((e_{\rm sc}, h_{\rm sc})\) be the electromagentic wave scattered by \(S\). Then the far--field pattern of \((e_{\rm sc}, h_{\rm sc})\) uniquely determines both the supporting hyperplane \(L\) and the screen \(S\) if neither \(p\) or \(q\) is parallel to \(\theta\).
\end{thm}

\medskip
The plan of this paper is as follows: Analysis of mathematical proof and concept for direct scattering problem of EM waves is discussed in Section one. We start by giving a precise definition of a planar screen and discuss time harmonic Maxwell's equations in the exterior of screen.
Representation theorem for the fundamental solution of  vector Helmholtz equation is also analysed here.
In Section 2 solution of the inverse problem is presented, including also the unique determination of the supporting hyperplane.\\
The paper is concluded with additional comments and future remarks.

\section{Scattering From a Perfectly Conducting Screen}

\subsection{Formal Definitions} 

\begin{defn}
A {\em planar \(C^k\)--screen}, \(k=1, \ldots, \infty\), in \(\mR^3\) is a compact, connected \(C^k\)--submanifold of an affine hyperplane \(L\subset \mR^3\). The affine hyperplane \(L\) is called {\em the supporting hyperplane} of \(S\). In the sequel we also fix a globally defined unit normal vector field on \(S\) and denote it by \(\nu\). Also, the boundary of \(S\) as a submanifold of \(L\) is denoted by \(\partial S\).
\end{defn}

\smallskip

Consider time--harmonic Maxwell's equations in the exterior of a screen \(S\):
\begin{equation}\label{Maxwell}
\nabla\times E  = i\omega\mu H , \quad \nabla\times H  = -i\omega\varepsilon E  \quad \mathrm{ in\, \,  } \mR^3 \setminus S.
\end{equation}
Here the {\em magnetic permeablity}� \(\mu\) and the {\em dielectricity}  \(\varepsilon\) are known positive constants, and we assume also that the angular frequency \(\omega>0\) is known.
\medskip

Given an incident filed \((E_0, H_0)\), i.e. a solution of 
\[
\nabla\times E_0  = i\omega\mu H_0 , \quad \nabla\times H_0  = -i\omega\varepsilon E_0  \quad \mathrm{ in\, \,  } \mR^3,
\]
the corresponding {\em scattered field} \((E_{sc}, H_{sc})\) is (formally) defined by demanding that \((E,H)\), where \(E = E_0 + E_{sc}\) and  \(H = H_0 + H_{sc}\), satisfy (\ref{Maxwell}), and the scattered field is outgoing in the sense that it satisfies the Silver--M\"uller--radiation conditions,
\begin{equation}\label{rad}
\hat r\times E_{sc} + \sqrt{\frac{\varepsilon}{\mu}} H_{sc} = o(|x|^{-1}), \quad \hat r\times H_{sc} - \sqrt{\frac{\mu}{\varepsilon}} E_{sc} = o(|x|^{-1}), \quad \text{as \(|x|\to \infty\)}.
\end{equation}
Here \(\hat r = x/|x|\). If we further assume that the screen is perfectly conducting, i.e the total field vanishes on \(S\), this leads to the {\em direct scattering problem for the perfectly conducting screen} \(S\): for a given incident field \((E_0, H_0)\) show that there is a unique scattered field \((E_{sc}, H_{sc})\) s.t. 
\begin{equation}\label{Maxwell-sc}
\nabla\times E_{sc}  = i\omega\mu H_{sc} , \quad \nabla\times H_{sc}  = -i\omega\varepsilon E_{sc}  \quad \mathrm{ in\, \,  } \mR^3 \setminus S
\end{equation}
satisfying \eqref{rad} and such that
\begin{equation}\label{screen cold}
\nu\times (E_{sc} + E_0) = 0 \quad \text{on \(S\)}.
\end{equation}
Note that we have not specified in what sense the boundary value \eqref{screen cold} holds. This will depend on the spaces where we look for solutions and the availability of suitable trace theorems. 
\medskip

\subsection{Representation Theorems} Assume for now that \(S\subset \mR^3\) is a \(C^2\)--screen. Denote by \(C^k_S\) the closed subspace of \(C^k(\mR^3\setminus S)\) consisting of  those \(u\in C^k(\mR^3\setminus S)\) such that \(u\) and all its derivatives  up to order \(k\) have normal limits on \(S\), i.e for all \(|\alpha|\leq k\) there are limits
\[
\lim_{\delta \to +0} \partial^\alpha_x u(x \pm \delta \nu(x)) = u^\pm_\alpha (x), \quad x\in S,
\]
where \(u_\alpha^\pm \in C(S)\). Note that we do not assume that limits \(u_\alpha^+\) and \(u_\alpha^-\) coincide on \(S\).

\begin{prop}\label{repthm}
Assume \((e,h)\in (C^1_S)^3 \times  (C^1_S)^3\) solves
\[
\nabla\times e  = i\omega\mu \, h , \quad \nabla\times h  = -i\omega\varepsilon \, e \quad \mathrm{ in\, \,  } \mR^3 \setminus S,
\]
and the Silver--M\"uller--radiation condition
\[
\hat r\times e + \sqrt{\frac{\varepsilon}{\mu}} h = o(|x|^{-1}), \quad \hat r\times h - \sqrt{\frac{\mu}{\varepsilon}} e = o(|x|^{-1}), \quad \text{as \(|x|\to \infty\)}, \quad \hat r = x/|x| >0.
\]
Then for all \(x\in \mR^3\), 
\ba
e(x) & = & \nabla \times \int_{S} \Phi(x-y) (\nu \times \{e^+(y) - e^-(y)\})\, ds(y)\\ & - & \frac{1}{i\omega\varepsilon} (\nabla \times)^2 \int_{S} \Phi(x-y) (\nu \times \{h^+ (y) - h^-(y)\})\, ds(y)
\ea
and
\ba{equation}
h(x) & = & \nabla \times \int_{S} \Phi(x-y) (\nu \times \{h^+(y) - h^-(y)\})\, ds(y)\\ & + & \frac{1}{i\omega\mu} (\nabla \times)^2 \int_{S} \Phi(x-y) (\nu \times \{e^+ (y) - e^-(y)\})\, ds(y).
\ea

\end{prop}

\begin{proof}
For \(\delta >0\) let \(\O _\delta = \{x \pm t\nu (x); \, x\in S, \, 0\leq t <\delta\}\) be a collar neighbourhood of \(S\). For \(\delta\) small enough this is a bounded piecewise analytic domain.
The standard representation formulas (see for example \cite{ck1}) give that for all \(x\in \mR^3\setminus \overline \O_\delta\) we have
\[
e(x) = \nabla \times \int_{\partial \O_\delta} \Phi(x-y) (\nu_\delta (y) \times e(y))\, ds(y) -\frac{1}{i\omega\varepsilon} (\nabla \times)^2 \int_{\partial \O_\delta} \Phi(x-y) (\nu_\delta (y) \times h(y))\, ds(y)
\]
and
\[
h(x) = \nabla \times \int_{\partial \O_\delta} \Phi(x-y) (\nu_\delta (y) \times h(y))\, ds(y) + \frac{1}{i\omega\mu} (\nabla \times)^2 \int_{\partial \O_\delta} \Phi(x-y) (\nu _\delta(y) \times e(y))\, ds(y).
\]
Here \(\Phi\) is the outgoing fundamental solution of the Helmholtz--operator \(\Delta + k^2\) and \(\nu_\delta\) is the exterior unit normal of \(\partial \O_\delta\). Then get as \(\delta \to +0\)
\ba
e(x) & = & \nabla \times \int_{S} \Phi(x-y) (\nu \times \{e^+(y) - e^-(y)\})\, ds(y)\\ & - & \frac{1}{i\omega\varepsilon} (\nabla \times)^2 \int_{S} \Phi(x-y) (\nu \times \{h^+ (y) - h^-(y)\})\, ds(y)
\ea
and
\ba
h(x) & = & \nabla \times \int_{S} \Phi(x-y) (\nu \times \{h^+(y) - h^-(y)\})\, ds(y)\\ & + & \frac{1}{i\omega\mu} (\nabla \times)^2 \int_{S} \Phi(x-y) (\nu \times \{e^+ (y) - e^-(y)\})\, ds(y),
\ea
as claimed.
\end{proof}

In what follows we will denote the jumps of a function (or a vector field) \(u\) across \(S\) by \([u]\), i.e.
\[
[u](y) = u^+(y) - u^-(y), \, \, y\in S.
\]

\medskip

\subsection{EM--plane waves and far--field patterns} Let \(\theta,\, p\in S^2\) and denote \(q= p\times \theta\). We call the field
\[
E(\theta; p,q) = \mu^{1/2} (p\times \theta)e^{ik\langle \theta, x\rangle}, \, H(\theta; p,q) = \varepsilon^{1/2} (q\times \theta) e^{ik\langle \theta, x\rangle}
\]
the {\em EM-plane wave with wavenumber \(k\), propagation direction \(\theta\) and polarizations \(p\) and \(q\),} It is easy to see that these fields satisfy the time harmonic Maxwell's equations
\[
\nabla \times E(\theta; p,q) = i\omega\mu \, H(\theta; p,q), \quad \nabla \times H(\theta; p,q) = -i\omega\varepsilon \, E(\theta; p,q).
\]
when \(k^2 = S^2 \varepsilon \mu\). Since the scalar components of the scattered electric- and magnetic fields are solutions of the Helmholz--equation \((\Delta + k^2) u = 0\) satisfying the Sommerfeld's radiation condition they have representations
\[
E_{\sc}(x) = \frac{E^\infty (x)}{|x|} + o (|x|^{-1}), \quad H_{\sc}(x) = \frac{H^\infty (x)}{|x|} + o (|x|^{-1}
\]
where \(E^\infty\) and \(H^\infty\) are the electric and magnetic far--fields patterns. If the initial field is the EM-plane wave \((E(\theta; p,q), H(\theta; p,q))\) we denote the corresponding far--field patterns by
\(E^\infty(\theta; p,q)\) and \(H^\infty (\theta; p,q)\).

\subsection{Relevant Sobolev--spaces}
Let (see \cite{H}, \cite{OPS}) \(L^2_{\rm loc} (\mR^3\setminus S)\) be the space of measurable functions which are square integrable on compact subsets of \(\mR^3\setminus S)\).
This becomes a Fr\'echet--space when equipped with semi--norms
\[
\| f\|_{R} = \| f\|_{L^2(\mR^3\setminus  S)\cap B_R(0))}, \quad R>R_0,
\]
where \(R_0\) is so large that \(S \subset B_{R_0}(0)\). Define also
\[
L^2_{\rm loc, \, curl} (\mR^3\setminus  S) = \{u\in L^2_{\rm loc}(\mR^3\setminus  S); \, \nabla\times u \in L^2_{\rm loc}(\mR^3\setminus  S)\},
\]
\[
L^2_{\rm loc, \, div} (\mR^3\setminus  S) = \{u\in L^2_{\rm loc}(\mR^3\setminus  S); \, \nabla\cdot u \in L^2_{\rm loc}(\mR^3\setminus  S)\},
\]
and equip these space with seminorms
\[
\|f\|_{R, {\rm curl} }= (\|f\|_R^2 + \| \nabla \times f\|_R^2)^{1/2}, \, \|f\|_{R, {\rm div} }= (\|f\|_R^2 + \| \nabla \cdot f\|_R^2)^{1/2}.
\]
Also, let 
\[
TH^{-1/2}(S) = \{u\in H^{-1/2}(S)^3; \langle \nu, u\rangle = 0\}
\]
i.e the space of tangential \(H^{-1/2}\)--fields on \(S\). We equip this with the norm induced from \(H^{-1/2}(S)^3\). With \({\rm Div}\) denoting the surface divergence we also define
\[
TH^{-1/2}_{\rm Div}(S) = \{u\in TH^{-1/2}(S); \Div (u) \in H^{-1/2}(S)\}
\]
and equip it with the Hilbert--norm defined by
\[
\| u\|_{TH^{-1/2}_{\rm Div}(S)}^2 = \|u\|_{TH^{-1/2}(S)}^2 + \| \Div (u)\|_{H^{-1/2}(S)}^2.
\]
Assume now that \(U\in \mR^3\) is a bounded \(C^2\)--domain having connected complement such that \(S\subset \partial U\) as a compact \(C^2\)--submanifold and we fix the unit normal \(\nu\) of \(S\) so that it extends to a unit exterior normal \(\tilde\nu\) of \(U\). If \(u, \, \phi\in (C^\infty_0)^3\) then vector Green's identities give
\[
\int_{\partial U} \langle \tilde\nu \times u, \phi\rangle\, ds = \int_U \langle \nabla \times u, \phi\rangle - \langle u, \nabla\times \phi\rangle\, dx,
\]
and extending this by density to \(\varphi\in H^1(\mR^3)\) and \(u\in L^2_{curl} (U)\) gives the existence of the tangential trace \(\tilde\nu\times u|_{\partial U} \in TH^{-1/2}(\partial U)\). We can argue similarly for the exterior domain. Using this definition we have a well--defined tangential trace maps \(t^\pm\) from the direction of \(\pm \nu\),
\[
t^\pm: L^2_{\rm loc, {\rm curl}} (\mR^3\setminus S) \ni u\mapsto \nu\times u^\pm \in TH^{-1/2}(S)
\]
Similarly, if \(u\in (C^\infty_0)^3\), \(\psi\in C^\infty_0\) we get from the Divergence--theorem
\[
\int_U \langle \tilde\nu, u\rangle \, \psi \, ds = \int_U\langle u, \nabla \psi\rangle + \psi \nabla\cdot u\, dx,
\] 
and using this we have a well--defined normal traces \(n^\pm\), 
\[
L^2_{\rm loc, \, div} (\mR^3\setminus  S) \ni u\mapsto \langle \nu, u^\pm\rangle \in H^{-1/2}(S).
\]
Note also that if \(u\in L^2_{\rm loc, {\rm curl}} (\mR^3\setminus S)\), then \(\nabla\times u\in L^2_{\rm loc, \, div} (\mR^3\setminus  S)\) and 
\[
{\rm Div}\, (\nu\times u^\pm) = -\langle \nu, \nabla \times u^\pm\rangle \in H^{-1/2}(S),
\]
i.e the tangential traces of \(L^2_{\rm loc, {\rm curl}} (\mR^3\setminus S)\) are in \(TH^{-1/2}_{\rm Div} (S)\). Note also that since extension by zero across a \(C^2\)--hypersurface is continuous in fractional Sobolev--spaces \(H^s\) when \(s<1/2\) the space \(TC^\infty_0(S)\) is dense in \(TH^{s}(S)\). However, this  is not necessarily true for the Div--spaces, and hence the closure of \(TC^\infty_0(S)\) in \(TH^{-1/2}_{\rm Div}(S)\) is denoted by \(\dot{TH}^{-1/2}_{\rm Div}(S)\).
\medskip

\subsection{Layer Potentials in Sobolev--spaces}\label{sobolev}
For \(x\in \mR^3\setminus S)\) and \(u\in C^\infty_0 (S)^3\) define the (vector) single layer potential of \(u\) by
\[
V_{\mR^3\setminus S} (u)(x) = \int_S \Phi (x-y) \, u(y)\, ds(y)
\]
and electromagnetic layer operators then by
\[
K_{\mR^3\setminus S} (u)(x) = \nabla \times V_{\mR^3\setminus S} (u)(x), 
\]
and
\[
N_{\mR^3\setminus S} (u)(x) = (\nabla \times)^2 V_{\mR^3\setminus S} (u)(x).
\]

\begin{prop}
Assume that \(S\) can be extended to a boundary \(\partial U\) for some bounded \(C^2\)--domain \(U\). Then the single layer potential has an extension to a bounded map
\[
V_{\mR^3\setminus S}: H^{-1/2}(S) \to H^1_{\rm loc}�(\mR^3\setminus S)
\]
and \(V_{\mR^3\setminus S}(u)\) satisfies the Sommerfeld--radiation condition. Also,  the electromagentic potentials have extensions into bounded maps
\[
K_{\mR^3\setminus S}, \, N_{\mR^3\setminus S}: \dot{TH}_{\Div}^{-1/2}(S) \to L^2_{\rm loc, curl}(\mR^3\setminus S)
\]
and  \(K_{\mR^3\setminus S} (u)\) and \(N_{\mR^3\setminus S} (u)\) satisfy the Sommerfeld--radiation conditions for any \(u\in  \dot{TH}_\Div^{-1/2}(S)\).
\end{prop}

\begin{proof} By known continuity properties (see for example \cite{McL}) the single layer defines a continuous map \(H^{-1/2}(U) \to H^1_{\rm loc}(\mR^3\setminus \overline U)\), and infact \(V_{\mR^3\setminus\overline U}(\phi)\) is continuous across \(U\) and the jump in the normal derivative is equal to \(\phi |_U\). Hence the claim for \(V\) follows since \(C_0^\infty (S)\) is dense in \(H^{-1/2}(S)\). This also implies the claims for \(K_{\mR^3\setminus S}\) and \(N_{\mR^3\setminus S}\) since for \(u\in TC^\infty_0(S)\) we have
\[
(\nabla \times)^2 V_{\mR^3\setminus S}(u) = -\nabla V_{\mR^3\setminus S}({\rm Div}\, u) + k^2 V_{\mR^3\setminus S}(u)
\]
and
\[
(\nabla \times)^3 V_{\mR^3\setminus S}(u) = k^2 \nabla\times  V_{\mR^3\setminus S}(u).
\]
\end{proof}

Using this we can generalize the representation theorem \ref{repthm} to weak solutions:

\begin{prop}\label{weak-repthm}
Let \(S \subset \mR^3\) be a \(C^1\)--screen. Assume \((e,h)\in L^2_{\rm loc, curl}(\mR^3\setminus S) \times  L^2_{\rm loc, curl}(\mR^3\setminus S)\)
solves
\[
\nabla\times e  = i\omega\mu \, h , \quad \nabla\times h  = -i\omega\varepsilon \, e \quad \text{ in\, \,  } \mR^3 \setminus S,
\]
and the Silver--M\"uller--radiation condition
\[
\hat r\times e + \sqrt{\frac{\varepsilon}{\mu}} h = o(|x|^{-1}), \quad \hat r\times h - \sqrt{\frac{\mu}{\varepsilon}} e = o(|x|^{-1}), \quad \text{as \(|x|\to \infty\)}, \quad \hat r = x/|x| >0.
\]
If \(\nu\times [e], \, \nu\times [h]\in TH^{-1/2}_{\rm Div}(S)\), then in \(\mR^3\setminus \overline S\), 
\[
e =  K_{\mR^3\setminus \overline S}(\nu\times [e])  - \frac{1}{i\omega\varepsilon}N_{\mR^3\setminus \overline S}(\nu\times [h]),
\]
and
\[
h =  K_{\mR^3\setminus \overline S}(\nu\times [h]) + \frac{1}{i\omega\mu}N_{\mR^3\setminus \overline S}(\nu\times [e]).
\]
Here \(\nu\) is the specified unit normal of \(S\). 
\hfill \(\Box\)

\end{prop}
\medskip

\subsection{Representation Formulas for the Scattered Field}

\begin{prop}\label{screen-rep} Let \(S\) be a perfectly conducting \(C^2\)--screen, and let 
\[
(E_{sc}, H_{sc})\in L^2_{\rm loc, curl} (\mR^3\setminus S) \times L^2_{\rm loc, curl} (\mR^3\setminus S) 
\] 
be the scattered field corresponding to an incoming field \((E_0, H_0)\). Then in \(\mR^3\setminus \overline S\) one has
\[
E_{\rm sc} = -\frac{1}{i\omega\varepsilon} N_{\mR^3\setminus \overline S}(\nu\times [H_{\rm sc}])
\]
and
\[
H_{\rm sc}  = K_{\mR^3\setminus \overline S}(\nu\times [H_{\rm sc}]).
\]
These fields have the following asymptotic behaviour as \(|x|\to \infty\):
\[
E_{\rm sc} (x) = -\hat x\times \left(\hat x \times \frac{e^{ik|x|}}{4\pi i\omega\varepsilon |x|}\int_S e^{-ik\langle \hat x, y\rangle} (\nu\times [H_{\rm sc}])(y) \, ds(y)\right) + O(|x|^{-2}), 
\]
\[
H_{\rm sc} (x) = \hat x \times \frac{e^{ik|x|}}{4\pi i\omega\mu |x|}\int_S e^{-ik\langle \hat x, y\rangle} (\nu\times [H_{\rm sc}])(y) \, ds(y) + O(|x|^{-2}), 
\]
where \(\hat x = x/|x|, \, x\not = 0\).
\end{prop}

\begin{proof}
Since on a perfectly conducting screen \(\nu\times [E_{\rm sc}] = -\nu\times [E_0] = 0\) the representations of \(E_{\rm sc}\) and \(H_{\rm sc}\) follow from Proposition \ref{repthm}. The asymptotic behaviour is obvious since for \(|x|\to \infty, \, y\in S\) and \(a(y)\) a vector field on \(S\), 
\[
\nabla_x \times \frac{e^{ik|x-y|}}{|x-y|}�a(y) = \hat x \times\frac{e^{ik|x|- ik\langle \hat x, y\rangle}}{|x|} a(y)  + O(|x|^{-2})
\]
\[
(\nabla_x \times)^2 \frac{e^{ik|x-y|}}{|x-y|}�a(y) = \hat x�\times (\hat x \times\frac{e^{ik|x|- ik\langle \hat x, y\rangle}}{|x|} a(y))  + O(|x|^{-2}).
\]
\end{proof}
\noindent In view of the above proposition we can write
\[
E_{\rm sc}(x) = \frac{e^{ik|x|}}{4\pi |x|} E^\infty (\hat x)  + O(|x|^{-2}), \quad H_{\rm sc}(x) = \frac{e^{ik|x|}}{4\pi |x|} H^\infty (\hat x)  + O(|x|^{-2})
\]
where the far--fields patterns \(E^\infty\) and \(H^\infty\) are given by
\[
E^\infty(\hat x) = -\hat x\times \left(\hat x \times \frac{1}{i\omega\varepsilon}\int_S e^{-ik\langle \hat x, y\rangle} (\nu\times [H_{\rm sc}])(y) \, ds(y)\right)
\]
\[
H^\infty(\hat x) = \hat x \times \frac{1}{i\omega\mu}\int_S e^{-ik\langle \hat x, y\rangle} (\nu\times [H_{\rm sc}])(y) \, ds(y).
\]
Note also that \(\varepsilon E^\infty (\hat x) = -\mu\, \hat x\times H^\infty (\hat x)\), and that the uniqueness of the scattered field follows from the uniqueness of the Dirichlet-- and Neumann--problems for the scalar Helmholtz--equation \cite{st}.
\medskip

\subsection{Integral Equations for the Scattered Field}
Assume now that \(U\subset \mR^3\) is a bounded \(C^2\)--domain having a connected complement such that \(S \subset \partial U\) is a \(C^2\)--submanifold. By the usual jump--relations (see for example \cite{ck1}, \cite{OPS}) tangential components of \(N_{\mR^3 \setminus\overline U} (u)\) and \(N_{U}(u)\) are continuous up to \(\partial U\) and they have equal traces for all \(u\in TH^{-1/2}_{\rm Div}(\partial U)\). Futhermore, for \(u\in TH^{-1/2}_{\rm Div}(\partial U)\) one has
\[
\nu\times N_{\mR^3 \setminus \overline U} (u)|_{\partial U} = \nu\times N_{U} (u)|_{\partial U} = N(u),
\]
where the surface integral operator \(N\) is given by
\[
N(u) = \tilde \nu \times \nabla S({\rm Div}\, u) + k^2 \tilde\nu \times S(u).
\]
Here \(\tilde \nu\) is the exterior unit normal to \(U\) which is assumed to agree with \(\nu\) on \(S\), and \(S\) is the direct boundary value of the single layer i.e 
\[
S(f)(x) = \int_{\partial U} \Phi (x-y) \, f(y)\, ds(y), \quad x\in \mR^3.
\]
Note also that by trace theorems given in subsection \ref{sobolev} on has \(N: TH^{-1/2}_{\rm Div} (\partial U) \to TH^{-1/2}_{\rm Div} (\partial U)\) continuously, and for the restriction to \(S\) one has
\(N: \dot{TH}^{-1/2}_{\rm Div} (S) \to {TH}^{-1/2}_{\rm Div} (S)\) again continuously. 

Assuming now that \((E_{sc}, H_{sc})\) is the field scattered by the screen \(S\) and \(\nu \times [H_{sc}]\in \dot{TH}^{-1/2}_{\rm Div}(S)\) we get from the Proposition \ref{screen-rep} and the continuity results of subsection \ref{sobolev} that
\[
\nu\times E_{sc}�= iN(\nu\times [H_{sc}])/\omega\varepsilon,
\]
and since the screen is perfectly conducting one gets an integral equation for the jump of the tangential component of the magnetic field,
\begin{equation}\label{IE}
-\nu\times E_0�= iN(\nu\times [H_{sc}])/\omega\varepsilon, \quad \nu\times [H_{sc}]\in \dot{TH}^{-1/2}_{\rm Div}(S).
\end{equation}
Solvability properties of this equation have been considered in \cite{bc}. More precisely, the showed that 
\eqref{IE} is uniquely solvable\footnote{This needs to be checked - they use a scale of function spaces tailored to the variational approach used, and also they do not assume orientablity of the screen, i.e for example a M\"obius--strip would be ok.} in \(\dot{TH}^{-1/2}(S)\)
\bigskip

\section{Solution of the Inverse Problem} 
\medskip

\subsection{Uniqueness when the supporting hyperplane is konwn.} 
The following lemma shows that for a planar screen the tangential density of the far--field pattern is uniquely determined.

\begin{lemma}
Assume that \(\rho\) is  a compactly supported tangential distributional density on a hyperplane \(L\). Let
\[
\rho^\infty (\hat x) = \hat x \times \langle \rho, \exp\{-ik\langle \hat x, \cdot\rangle\}\rangle, \quad \hat x \in \Ss^2.
\]
Then the map \(\rho \mapsto \rho^\infty\) is injective. 
\end{lemma}

\begin{proof}
We may assume that coordinates have been chosen so that \(L\) is defined by 
\(\{x; \, x_3 = 0\}\). Let \(\rho = ad\sigma\) where \(a= (a_1, a_2) \in \E'(\mR^2)\) and \(d\sigma\) is the surface measure on the hyperplane \(L\).  Then \(\rho^\infty = 0\) is equivalent with
\[
\xi \times (\hat a_1 (\xi'), \hat a_2 (\xi'), 0) = 0,\, \xi = (\xi', (k^2 - |\xi'|^2)^{1/2}), \, |\xi'|<k,
\]
and hence \(\hat a_1\) and \(\hat a_2\) vanish in the unit ball of \(\mR^2\) and since they are entire functions they are identically zero.
\end{proof}

This implies that a the far--field \(E^\infty\) (or \(H^\infty\) for that matter) uniquely determines the density \([\nu\times H_{\rm sc}]\, ds\) when the screen \(S\) is flat, i.e. included in a hyperplane.

\begin{prop}
Let \(S\) be a \(C^2\)--screen contained in a supporting hyperplane \(L\) and let \((e_0, h_0)\) be an electromagnetic plane wave with wave number \(k\) and having electric and magnetic polarisations \(p\) and \(q\).
Assume that  \(\rho \in \dot{TH}^{-1/2}(S)\) solves \(-\nu\times e_0 = -iN( \nu \times \rho)/\omega\varepsilon\) on \(S\). Then if \(p\) or \(q\) are not parallel to \(\theta\) the density \(\rho\) has full support, i.e. \(\supp (\rho) = S\).
\end{prop}

\begin{proof} Assume coordinates chosen so that the \(L = \{x\in\mR^3; \, x_3 = 0\}\). Assume that there is a relatively open \(U\subset S\) such that \(\rho = 0\). Define
\[
\tilde h = K_{\mR^3 \setminus S} (\nu \times \rho), \quad \tilde e = \frac{i}{\omega \varepsilon}�\nabla \times \tilde h.
\]
Then \(\tilde h, \, \tilde e\in L^2_{\rm loc, \, curl} (\mR^2 \setminus S)\) and they satisfy Maxwell's equations
\[
\nabla \times \tilde e = i\omega \mu \, \tilde h, \quad \nabla \times \tilde h = -i\omega \varepsilon \,  \tilde e.
\]
The second equation follows from the definition and the first is an immediate consequence of vector Green's formulas: 
\ba
\nabla \times \tilde e & = & \frac{i}{\omega \varepsilon}�(\nabla \times)^2 \, \tilde h = \frac{i}{\omega \varepsilon}(\nabla\nabla\cdot - \Delta)(K_{\mR^3 \setminus S} (\nu \times \rho)) = \frac{i}{\omega \varepsilon}(\nabla\nabla\cdot - \Delta)(\nabla\times S_{\mR^3 \setminus S} (\nu \times \rho))\\
& = & \frac{i}{\omega \varepsilon} k^2 \, \nabla\times S_{\mR^3 \setminus S} (\nu \times \rho) = -i\omega \mu \nabla\times S_{\mR^3 \setminus S} (\nu \times \rho) = -i\omega \mu \, \tilde h.
\ea
Also by jump relations of the vector potentials \([\nu\times \tilde h] = [\nu\times \tilde \rho] = 0\) on \(U\) and 
\[
-\tilde e = -\frac{i}{\omega\varepsilon}\nabla\times \tilde h  =  -\frac{i}{\omega\varepsilon}N_{\mR^3\setminus S} (\nu\times \rho) 
\]
so that 
\[
\nu\times \tilde e|_S = iN(\nu\times \rho)/\omega\varepsilon|_S = \nu\times e_0.
\]
Let \(E = e_0 - \tilde e\) and \(H = h_0 - \tilde h\). Then by the above observations \((E,H)\) solves \eqref{Maxwell} and
\[
\nu\times E|_S = 0, \quad [\nu \times H]|_U = 0.
\]
Let now \(\hat E\) be an extension of \(E\) from the upper half--space \(\{x_3 >0\}\) to lower half--space so that it is odd in the tangential components and even in the normal component, i.e 
\[
\hat E (x_1, x_2, -x_3) = (-E_1 (x), -E_2(x), E_3 (x)), \,\, x = (x_1, x_2, x_3),
\]
and let \(\hat H\) be an extension of \(H\) which is even in tangential components and odd in the normal component,
\[
\hat H (x_1, x_2, -x_3) = (H_1 (x), H_2(x), -H_3 (x)), \,\, x = (x_1, x_2, x_3).
\]
Then a straightforward computation shows that
\[
\nabla \times \hat E = i\omega \mu \, \hat H, \quad \nabla \times \hat H = -i\omega \varepsilon \,  \hat E, \quad x_3 \not = 0.
\]
Let \(V = U\times \mR\). Then since the tangential components of \(E\) vanish on \(S\) and the tangential components of \(H\) are continuous across \(U\) this holds also for tangential components of \(\hat E\) and \(\hat H\) across \(U\), thus \((\hat E, \hat H)\) solves 
\[
\nabla \times \hat E = i\omega \mu \, \hat H, \quad \nabla \times \hat H = -i\omega \varepsilon \,  \hat E, \quad x\in V,
\]
and hence by unique continuation \(E = \hat E\) and \(H = \hat H\) in \(V\) and thus also in \(\mR^3 \setminus S\)  since \(\nu \times \hat H = \nu \times H\) on \(U\). We can hence write
\begin{equation}\label{even-odd}
e_0 = \hat E + \tilde e, \quad h_0 = \hat H + \tilde h.
\end{equation}
Lets define that a vector field has parity \(1\) if the tangential components are even and the normal component is odd with respect to \(\{x_3 = 0\}\), and it has parity \(-1\) if the tangential components are odd and the normal component is even. Notice then that since \(\tilde h\) is the EM-double layer of a tangential density, it has parity \(-1\), and hence \(\tilde e = -i\nabla \times \tilde h/\omega\varepsilon\) has parity \(1\). Also, the decomposition of a field as a sum of fields with parity \(+1\) and \(-1\) is unique. Since \(\tilde e\) satisfies the Silver--M\"uller radiation condition, the incoming field \(e_0\) must have parity \(-1\) and similarly \(h_0\) must have parity \(1\). Recall that
\[
e_0 (x) = \mu^{1/2} (p\times \theta)e^{ik\langle \theta, x\rangle}, \, h_0(x)  = \varepsilon^{1/2} (q\times \theta) e^{ik\langle \theta, x\rangle}, \quad q = p\times \theta,
\] 
Hence the parity \(1\) part of \(e_0\) is given by
\[
\mu^{1/2}(q_1 e_0^{(+)}, q_2 e_0^{(+)}, iq_3 e_0^{(-)}),
\] 
where \(e_0^{(+)}(x) = \cos (k\langle\theta, x\rangle)\) and \(e_0^{(-)}(x) = \sin (k\langle\theta, x\rangle)\). This vanishes identically if and only if \(q = 0\), i.e \(p\times \theta = 0\). Similarly, the parity \(-1\) part of \(h_0\) is 
\[
\varepsilon^{1/2}(i(q\times\theta)_1 e_0^{(-)}, i(q\times \theta)_2 e_0^{(-)}, (q\times\theta)_3 e_0^{(+)}),
\] 
which vanishes identically if and only if \(q\times \theta = 0\). Since \(p\), \(q\) and \(\theta\) are unit vectors and \(q= p \times \theta\) this is not possible. \(\Box\)
\end{proof}
\medskip

\subsection{Unique Determination of a Planar Screen}

We show that the supporting hyperplane is uniquely determined the the far--field of a single scattering solution. This combined with the unique determination results of the previous subsection then proves Theorem \ref{mainthm}.

\begin{prop}
Assume \(S_1\) and \(S_2\) are two planar screens contained in supporting hyperplanes \(\pi_1\) and \(\pi_2\) respectively. Assume \(u_1 = (e_1, h_2)\) and \(u_2 = (e_2, h_2)\) scattering solutions for the screens \(S_1\) and \(S_2\) corresponding to the same initial field and having equal non--vanishing far fields. Then \(\pi_1 = \pi_2\).
\end{prop}

\begin{proof} Let \(\rho_1\) and \(\rho_2\) be the jumps of \(\nu_1 \times h_1\) and \(\nu_2 \times h_2\) across \(S_1\) and \(S_2\) respectively. Here \(\nu_i\) is the specified unit normal to \(S_1\). Since \(u_1\) and \(u_2\) have equal far fields and \(\mR^3 \setminus (S_1\cup S_2)\) is connected, we must have \(u_1 = u_2\) there. Hence both fields must be smooth across \((S_1 \cup S_2 ) \setminus (S_1 \cap S_2)\) i.e both densities \(\rho_1\) and \(\rho_2\) are supported in the intersection \(S_1\cap S_2\). If the planes \(\pi_1\) and \(\pi_2\) intersect transversally, the jumps are supported on a codimension \(2\) subspace, and since they belong \(\dot{\rm TH}^{-1/2}(S_1 \cup S_2)\) they must vanish\footnote{Note that a nonvanishing compactly supported distribution density on a codimension \(2\) submanifold of \(\mR^3\)  belongs to \(H^s\) if and only if \(s<-1\). This follows for example from estimates at the end of Section 7.1 in \cite{HoI} by applying these to a suitable dyadic decomposition.} if the intersection is transversal so also the far fields vanish. 
\end{proof}

\end{document}